%% file: gale2.tex
\documentclass[11pt]{article}
\usepackage{anysize,url,amsmath,amssymb,amsthm,paralist}
\usepackage{psfrag,times,multicol,graphicx,color}
\newtheorem{theorem}{Theorem}
\newtheorem{lemma}[theorem]{Lemma}

\newtheorem{proposition}[theorem]{Proposition}
\newtheorem{corollary}[theorem]{Corollary}
\theoremstyle{definition}
\newtheorem{definition}[theorem]{Definition}
\newtheorem{example}[theorem]{Example}
\theoremstyle{remark}

\newtheorem*{oexercise}{Optimization Exercise}
\newcommand{\R}{{\mathbb R}}
\newcommand{\Q}{{\mathbb Q}}
\newcommand{\Z}{{\mathbb Z}}
\newcommand{\CC}{{\mathcal C}}
\newcommand{\VV}{{\mathcal V}}
\newcommand{\conv}{\mathop\mathrm{conv}}
\newcommand{\vv}{v}
\newcommand{\vb}{\bar v}
\newcommand{\vbb}{\bar {\bar v}}
\title{Non-rational configurations, polytopes, and surfaces}

\author{\Large G\"unter M. Ziegler\\
{Inst.\ Mathematics, MA 6-2, TU Berlin}\\
{D-10623 Berlin, Germany}\\
\url{ziegler@math.tu-berlin.de}} 

\date{October 24, 2007; revised November 16, 2007}
\begin{document}

\maketitle

\begin{center}
  dedicated to Micha Perles
\end{center}

\begin{multicols}{2}
The two most interesting ``platonic solids'', the
regular dodecahedron 
\begin{center}
  \includegraphics[height=50mm]{dodecahedron}
\end{center}
and the regular icosahedron,
\begin{center}
  \includegraphics[height=50mm]{icosahedron}
 \end{center}
necessarily have \emph{irrational vertex coordinates}.

Indeed, they involve regular pentagons, as faces
(in the dodecahedron), or given by the five neighbors of any vertex
(for the icosahedron); and a regular pentagon cannot be realized with
rational coordinates, since the diagonals intersect each other in 
the ratio $\tau:1$ known as the ``golden section'' \cite[p.~30]{Coxeter},
where $\tau=\frac12(1+\sqrt5)$.

However, the dodecahedron and the icosahedron can be realized
with rational coordinates if we do not require them to be precisely regular:
If you perturb the vertices of a regular icosahedron ``just a bit''
into rational position, then taking the convex hull will 
yield a rational polytope that is \emph{combinatorially equivalent}
to the regular icosahedron.
Similarly, by perturbing the facet planes of a regular dodecahedron
a bit we obtain a dodecahedron with rational coordinates.

Indeed, \emph{every} combinatorial type of $3$-dimensional polytope 
can be realized with rational coordinates.
For simplicial polytopes such as the icosahedron, where all faces are
triangles, this can be achieved
by perturbing vertex coordinates. For simple polytopes such as the 
dodecahedron, where all vertices have degree three, we can 
perturb the planes spanned by faces into rational position
(that is, until the planes have equations with rational coefficients).
For the case of general 3-polytopes, which may be neither simple
nor simplicial, the result is not obvious, but we get it
as an easy consequence of Steinitz's proof for
his (deep) theorem \cite{Stei1,StRa} \cite[Lect.~4]{Z35} that
every 3-connected planar graph is the graph of a convex polytope.

In view of this, it is a surprising and perhaps counter-intuitive
discovery, made by Micha Perles in the sixties, that in high
dimensions there are inherently \emph{non-rational} combinatorial
types of polytopes: Specifically, Perles constructed an $8$-dimensional
polytope with $12$ vertices 
that can be realized with vertex coordinates in $\Q[\sqrt5]$, but not
with rational coordinates.  His construction was given in terms of
``Gale diagrams'', which he introduced and developed into a powerful
tool for the analysis of polytopes with ``few vertices'', that is,
$d$-dimensional polytopes with $d+b$ vertices for small~$b$.%
\footnote{Micha
Perles, a professor of mathematics at Hebrew University in Jerusalem
who just retired, is
a remarkable mathematician who has published very little, but
contributed a number of brilliant ideas, concepts, and proofs.
His theory of Gale diagrams, as well as his construction of non-rational
polytopes, were first published in the 1967 first edition of Branko
Gr\"unbaum's book ``Convex Polytopes'' \cite{Gr1-2}.
(See \cite{perles84:_at_e} or \cite[Chap.~13]{Z58-3} for another gem.)}

Gale diagrams are a duality theory:
They involve the passage to a space of complementary dimension
(for a $d$-polytope with $n$ vertices one arrives 
at an investigation in $\R^{n-d-1}$), and so 
the polytopes produced by Perles' construction are hard to ``visualize''.
However,  it was later found that non-rational
polytopes may be generated from planar (non-rational)
incidence configurations in a number of different ways,
the simplest of which are ``Lawrence extensions''.
These were discovered and used, but not published, in 1980 by Jim Lawrence,
then at the University of Kentucky; they first appeared in print in a paper
by Billera \& Munson~\cite{BiMu} on oriented matroids.
Lawrence extensions may be described via two dualization processes,
but two dualizations are as good as none: and so we
arrive at a ``direct'' construction in primal space~\ldots

As you will see below, granted that non-rational point configurations
in the plane exist (which we will see), Lawrence
extensions are almost trivial to perform, and quite easy to analyze.

One might try to attribute all this to the fact that
``high-dimensional geometry is weird''. However, although there is some
truth to this claim, the fact that non-rational planar incidence configurations
lead to non-rational geometric structures may also be seen in
other instances. So, we sketch a construction by
Ulrich Brehm, announced in 1997 \cite{brehmOW2} but not yet published in full, yet,
which shows that there are geometric objects in $\R^3$ (namely,
certain polyhedral surfaces) that are intrinsically non-rational.

Constructing instances of non-rational polytopes, or of
non-rational surfaces, is not hard with the techniques we have at hand.
Since the analysis and proofs become quite easy if we work with homogeneous
coordinates (that is, in projective geometry), we will review this 
tool first; note that it is not used in the constructions.

Much harder work --- both in the careful statement of the results, and
in the proofs of the theorems --- is needed if one is striving
for so-called universality theorems; these say that the configuration
spaces of various geometric objects ``are arbitrarily wild''.
We will have a brief discussion later in this paper,
before we 
end with major open problems.

\begin{small}
\subsection*{Acknowledgements.}
Thanks to Volker Kaibel for the discussions and joint drafts
on the path to this article, to Nikolaus Witte for many comments
and some of the pictures, to John M. Sullivan and Peter McMullen for careful
and insightful readings, to Ravi Vakil and Michael Kleber for 
their encouragement and guidance on the way towards publication in
the \textsl{Math.\ Intelligencer}, and in particular to Ulrich Brehm for his 
permission to report about his mathematics ``to be published''.
\end{small}

\section*{Homogeneous coordinates and\\ projective transformations}

An abstract \emph{configuration} is given by a set $\{p_1,\dots,p_n\}$ of 
$n$ elements (``points'') and by a list which says which
triples of points should be \emph{collinear} (and that the others shouldn't).
A \emph{realization} of the configuration is given by $n$ points
$w_1=(x_1,y_1)$, \ldots, $w_n=(x_n,y_n)\in\R^2$ that satisfy the conditions,
under the correspondence $p_i\leftrightarrow w_i$:
The points $w_i$, $w_j$, and $w_k$ should be collinear
exactly if this had been dictated for $p_i,p_j,p_k$. 
``Being collinear'' is a linear
algebra condition for $w_i,w_j,w_k$: The points $w_i$, $w_j$, and $w_k$
need to lie on a line, that is,
be \emph{affinely dependent}. Equivalently, the vectors
 $(1,x_i,y_i)$, $(1,x_j,y_j)$, $(1,x_k,y_k)\in\R^3$ need to be
\emph{linearly dependent}, that is, have determinant zero.
Every realization by points $w_i$ in~$\R^2$ corresponds to a
realization by vectors $v_i:=(1,x_i,y_i)$ in~$\R^3$. These coordinates
with a first coordinate $1$ prepended are referred to as 
\emph{homogeneous coordinates}.

All of what follows in this paper could in principle be discussed 
(and computed) in affine coordinates --- it would just be much 
more complicated. 

A key observation is now that linear independence is not affected if we
replace any one of the vectors $v_i\in\R^3$ by a non-zero multiple.
\\
Here are four fundamental facts.
\begin{compactitem}[$\bullet$]
\item Any realization by points $w_i\in\R^2$ and specified affinely
  dependent triples yields a realization by vectors $v_i\in\R^3$
  with specified linearly dependent triples: just pass to homogeneous
  coordinates.
\item Conversely, any ``linear'' realization by vectors $v_i\in\R^3$ can be
  converted into an ``affine'' realization by points in~$\R^2$, by
  \emph{dehomogenization}: Find a plane $a t + b x+ cy=1$ that is not
  parallel to any one of the vectors, and rescale the vectors to lie
  on the plane.  (That is, find a linear function $\ell(t,x,y)=a t + b
  x+ cy$ that does not vanish on any one of the vectors, and then
  replace $v_i$ by $\frac1{\ell(v_i)}v_i$.)
\item Invertible linear transformations on $\R^3$ correspond to 
  \emph{projective transformations} in the plane~$\R^2$.
\item Any four vectors $v_1,v_2,v_3,v_4\in\R^3$ such that no three of
  them are linearly dependent form a \emph{projective basis}: There is a unique
  projective transformation that maps them to $e_1$, $e_2$, $e_3$, and
  $e_1+e_2+e_3$, that is, a linear transformation that maps them to
  non-zero multiples of these four vectors.  Indeed, if
  $v_4=\alpha_1v_1+\alpha_2v_2+\alpha_3v_3$ with nonzero $\alpha_i$,
  then consider $\{\alpha_1v_1,\alpha_2v_2,\alpha_3v_3\}$
  as a basis and let the linear transformation map $\alpha_iv_i$ to~$e_i$.
\end{compactitem}
Clearly, the concepts of homogenization, dehomogenization, projective
transformations, and projective bases work analogously also for higher
dimensions. It's elementary real linear algebra.
\medskip

For the study of convex polytopes
it is also advantageous to treat them in homogenous
coordinates.  However, here convexity is important, and thus we
have to insist on the use of \emph{positive} rather than non-zero
coefficients/\allowbreak multiples throughout.
\[
\includegraphics{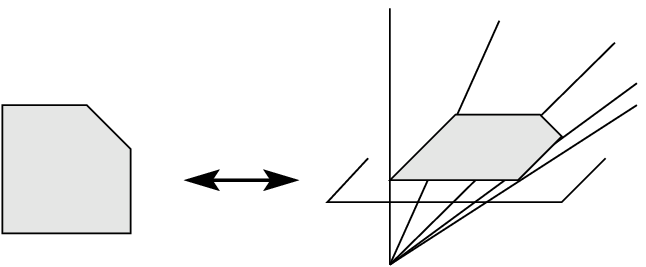}
\unitlength1cm
\begin{picture}(0,0)
\put(-6.4,.8){$P$}
\put(-1.5,2){$C_P$}
\put(-.6,.55){$t=1$}
\end{picture}
\]
In this setting, \emph{homogenization} is the
passage from a $d$-dimensional convex polytope
$P\subset\R^d$ with $n$ vertices to a $(d+1)$-dimensional pointed
convex polyhedral cone $C_P\subset\R^{d+1}$ with $n$ extreme rays.
More generally, any $k$-dimensional face of~$P$ corresponds to a
non-empty $(k+1)$-dimensional face of the cone $C_P$, and is thus 
supported by a \emph{linear} hyperplane 
through the origin, which is the apex of~$C_P$.
Dehomogenization allows us to pass back from any 
$(d+1)$-dimensional pointed polyhedral
cone to a $d$-polytope.  Moreover, rational $d$-polytopes correspond
to rational $(d+1)$-cones, and conversely(!). (See
e.g.\ \cite[Sect.~2.6]{Z35} for a more detailed discussion.)

\section*{Non-rational configurations}

The insight that there are ``abstract'' combinatorial incidence
configurations that may be geometrized with real, but not with
rational coordinates is rooted deep in the history of projective
geometry.  (A reason for this is that addition and multiplication can
be modelled by incidence configurations, via the von Staudt
constructions
\cite[2. Heft, 1857]{staudt56:_beitr_geomet_lage}; thus,
polynomial equations can be encoded into point configurations.
This mechanism was studied already very early in the 
framework of matroid theory, starting with MacLane's fundamental
1936 paper \cite{maclane36:_some},
\cite[pp.~147-151]{kung86:SourceBook}, 
where an eleven-point example was described.
See Kung \cite[Sect.~II.1]{kung86:SourceBook} for details and
further references.)
As suggested by Perles, 
let us look more closely at the regular pentagon.

\begin{example}\label{ex:pentagon}
  The \emph{extended pentagon configuration} $\CC_{11}$ is an abstract
  configuration on eleven points $p_1,\dots,p_{11}$: 
  Its collinear triples 
  $\{p_1,p_2,p_7\}$, $\{p_1,p_3,p_8\}$, \ldots\ may be read off from the 
  collinearities among the five vertices of a regular
  pentagon, the five intersection points of its diagonals, and the
  center.  There are ten lines that contain more than two
  of these points: The five diagonals of the pentagon contain four
  points each, while the five lines of symmetry contain three.
\end{example}

\begin{center}
\input pentagon1b.pstex_t
\end{center}

\noindent
Now we want to ``realize'' this configuration in the rational plane, that is,
find rational coordinates for all the eleven points, such that the collinearities
given by the ten lines are satisfied --- and such that the configuration does
not ``collapse'', that is, no further collinearities should occur.
(We will not check that latter condition in detail, but it is important:
In view of the next lemma check that there are rational
coordinates for the eleven points that satisfy all ten collinearities,
for example given by eleven distinct points on one line, or with the eleven points placed
at the vertices of a triangle.)

\begin{lemma}\label{lemma:pentagon}
  The eleven-point configuration of Example~\ref{ex:pentagon} can be
  realized with coordinates in $\Q[\sqrt5]$, but not with rational
  coordinates.
\end{lemma}

\begin{proof}
The calculation for this lemma is most easily done in terms of
homogeneous coordinates. 

In a {vector realization} $v_1,\dots,v_{11}\in\R^3$,
no three of the four vectors $v_1,v_2,v_9,v_{10}$ can be coplanar:
These four vectors form a projective basis.
Thus we can assume that they have, for example, the coordinates $v_1=(1,0,-1)$, 
$v_2=(1,0,1)$, $v_9=(1,-1,0)$, and $v_{10}=(1,1,0)$.
Furthermore, $v_3$ will have homogeneous coordinates $(1,a,0)$ for some
parameter $a\in\R{\setminus}\{-1,+1\}$ that we need to determine.
Now it is easy (exercise!) to derive coordinates for the other vectors
and equations for the lines they span, for example in the order
$\ell_1:x_2=0$, $\ell_2:x_1=0$, $\ell_3$, $\ell_4$, $\ell_5$, $\ell_6$,
$v_4=(0,1,-1)$, $v_5$, $\ell_7$, and then
$v_7=(1,0,-a)$, $v_8=(1-a,2a,1+a)$. Finally, the condition that
$v_4,v_7$ and $v_8$ need to be linearly dependent leads to the 
determinant equation $a^2-4a-1=0$, that is, $a=2\pm\sqrt5$.
\begin{center}
\input pentagon2b.pstex_t
\end{center}
You should do this computation yourself. To compare results, use
the labels in our figure.
\end{proof}

The eleven-point ``extended pentagon''
example is not minimal, as you are invited to find out
in the course of your computation.

\begin{oexercise}
  Show that the nine-point configuration obtained from deleting the
  points $p_6$ and $p_{11}$ also has the properties derived in
  Lemma~\ref{lemma:pentagon}.
\end{oexercise}

\section*{Non-rational polytopes}

Let $\CC$ be again a $2$-dimensional point configuration
consisting of $n$ points, and we assume that we have a
realization $\VV=\{v_1,\ldots,v_n\}$ of the configuration at hand. 
For the following, we should also
assume that all the points $v_i$ are distinct, that the 
$n$ points do not lie on one line, and that this holds ``stably'': If we
delete any one of the points, then the others should not lie
on a line. 

If $v\in\VV$ is any point in the configuration, a
\emph{Lawrence extension} is performed on~$v$ by replacing $v$
by two new points $\vb$ and~$\vbb$ on a line through $\vv$ that uses a new
dimension. That is, $\vv$, $\vb$ and $\vbb$ are to lie in this order on a
line $\ell$ that intersects the affine span of $\CC$ only
in~$\vv$.  Thus by this addition of two new points $\vb$ and $\vbb$ and
deletion of the ``old'' point~$\vv$, the dimension of a configuration
goes up by one, and so does the number of points.
We will iterate this, applying Lawrence extensions to all
points in the configuration $\CC$, one after the other.
\begin{center}
 \input lawrence_ext1.pstex_t
\end{center}

\begin{definition}
  The \emph{Lawrence lifting} $\Lambda\VV$ of an $n$-point configuration $\VV$ 
is obtained by successively applying Lawrence extensions 
to \emph{all} the $n$ points of~$\VV$. 
Thus the Lawrence lifting of a
$2$-\allowbreak dimensional $n$-point configuration $\VV$
is a $(2+n)$-\allowbreak dimensional configuration that consists of $2n$ points.
\end{definition}

Lawrence has observed that this simple construction has a number of
remarkable properties.  
First, the order in which the Lawrence extensions
are performed does not matter, since they use independent ``new''
directions.
This may also be seen from a coordinate representation: If the $n$
points of $\VV$ are given by $(x_1^i,x_2^i)\in\R^2$, then
$\Lambda\VV$ is given by the rows of the $2n\times (2+n)$ matrix
\[
\left(\begin{array}{c}
\vb_1 \\ \vb_2 \\ \vdots \\ \vb_n \\
\vbb_1\\ \vbb_2\\ \vdots \\ \vbb_n 
\end{array}
\right)\ \ :=\ \ 
\left(
  \begin{array}{ccccccccccccccc}
    x_1^1 & x_2^1 & 1 &   &       &   \\
    x_1^2 & x_2^2 &   & 1 &       &   \\
   \vdots &       &   &   &\ddots &   \\
    x_1^n & x_2^n &   &   &       & 1 \\ 
    x_1^1 & x_2^1 & 2 &   &       &   \\
    x_1^2 & x_2^2 &   & 2 &       &   \\
   \vdots &       &   &   &\ddots &   \\
    x_1^n & x_2^n &   &   &       & 2 
  \end{array}
\right).
\]

Here $\vb_i$ and $\vbb_i$ arise by lifting $v_i$ into a new
$i$-th direction; the specific values $1$ and $2$ for the 
``lifting heights'' are not important, other positive values
would give equivalent configurations.

Next, the points $\vb_1,\dots,\vb_n,\vbb_1,\dots,\vbb_n$ of
$\Lambda\VV$ are in convex position, so they are the vertices of a
polytope.  Moreover, for each $i$ the pair of
vertices $\vb_i,\vbb_i$ forms an edge of this polytope $\conv\Lambda\VV$.
Indeed, it suffices to verify the last claim: 
From now on, let us denote the coordinates on $\R^{2+n}$ by 
$(x_1,x_2,y_1,\dots,y_n)$. Among the points of
$\Lambda\VV$, the points $\vb_i$ and $\vbb_i$ minimize the linear
functional $(y_1+\dots+y_n)-y_i$, which sums all ``new
variables'' except for the $i$th one. Thus $e_i=[\vb_i,\vbb_i]$ is an edge
of~$\Lambda\VV$, and its endpoints are vertices.

\begin{definition}
The \emph{Lawrence polytope} of the realized configuration $\VV\subset\R^2$
is the convex hull of its Lawrence lifting,
\[L(\VV)\ \  :=\ \ \conv\Lambda\VV\ \ \subset\ \ \R^{2+n}.
\]
\end{definition}

The vertices not on $e_i$, i.e.\ the set $\Lambda\VV\setminus\{\vb,\vbb\}$,
form the vertex set of a facet $F_i$ of the polytope $L(\VV)$: This is
since they all minimize the linear functional $y_i$, and span a
Lawrence polytope of dimension $1+n$.

Finally, let $\ell$ be any line of the original $2$-dimensional
configuration, which contains the points $\vv_i$ ($i\in I^0$), and has the points
$\vv_j$ ($j\in I^-$) on one side, and the points
$\vv_k$ ($k\in I^+$) on the other side,
for a partition $I^0\cup I^-\cup I^+=[n]$. Then there is a facet
$F^\ell$ of $L(\VV)$ with vertex set $V(F^\ell)\ =\ $
\[ 
\big\{ \vb_j        : j\in I^- \big\} \cup
\big\{ \vb_i,\vbb_i : i\in I^0 \big\} \cup  
\big\{       \vbb_k : k\in I^+ \big\}.
\]
To see this, let $l(x_1,x_2)=ax_1+bx_2+c$ be a linear function that
is  zero on $\vv_i$ ($i\in I^0$), 
negative on $\vv_j$ ($j\in I^-$), and 
positive on $\vv_k$ ($k\in I^+$). From this we can easily
write down a functional 
\[
\bar l(x_1,x_2,y_1,\dots,y_n) := 
 l(x_1,x_2) 
\ +\  \alpha_1 y_1+\dots+\alpha_n y_n
\]
that is zero on the purported vertices of $F_\ell$,
and positive on all other vertices of $L(\VV)$: By just
plugging in, you are led to set
\[ \alpha_j\ :=\ 
\begin{cases}    0        & \textrm{ for } j\in I^0,\\
          -\phantom{\frac12}l(x_1^j,x_2^j) & \textrm{ for } j\in I^-,\\
 -\tfrac12 l(x_1^j,x_2^j) & \textrm{ for } j\in I^+.
\end{cases}
\]
Finally, we check that the face $F_\ell$ indeed
has dimension $1+n$, so it defines a facet of $L(\VV)$.

Clearly if a configuration $\VV$ has rational coordinates 
then so do the Lawrence lifting $\Lambda\VV$ (see the matrix above)
and the Lawrence polytope $L(\VV)$.
Lawrence's remarkable observation was that the converse is true as well.

\begin{theorem}[Lawrence]\label{thm:Lawrence}
  Any realization of the Lawrence polytope $L(\VV)$ encodes a
  realization of~$\VV$.  Thus, if $L(\VV)$ has rational coordinates,
  then so does~$\VV$.
\end{theorem}

\begin{proof}
Let $P\subset\R^{2+n}$ be a polytope with the combinatorial type of $L(\VV)$.
Somehow we have to start with $P$ and ``construct'' $\VV$ from it.

For this we homogenize, and let $C_P\subset\R^{3+n}$ be the
polyhedral cone spanned by~$P$.
Let $H_i\subset\R^{3+n}$ be the linear hyperplanes spanned
by the $n$ facets $F_i\subset P$ discussed above. 
The intersection 
\[
R\ \ :=\ \ H_1\cap\dots\cap H_n
\]
of these facets 
is a $3$-dimensional linear subspace of~$\R^{3+n}$: Indeed
the intersection of $n$ hyperplanes has co-dimension at most~$n$,
and the co-dimension cannot be smaller than $n$ since for each
$H_i$ there are vertices that are not contained in $H_i$, but 
in all the other hyperplanes $H_j$ (namely $\vb_i$ and $\vbb_i$).
The subspace $R$ is the \emph{space} 
where we will construct a vector representation of~$\VV$.

Let $E_i\subset\R^{3+n}$ be the $2$-dimensional linear subspace that is spanned
by the edge $e_i$ (that is, by $\vb_i$ and~$\vbb_i$). 
Now  $e_i$ is contained in $F_j$ for all $j\neq i$, but not in $F_i$; 
thus $E_i$ is contained in $H_j$ for all $j\neq i$, but not in~$H_i$.
So if we intersect $R$ with~$E_i$, we get a linear space
\[
R\cap E_i \ =\ H_i\cap E_i \ \ =:\ \ V_i
\]
that is $1$-dimensional. (The intersection of
a $2$-dimensional subspace with a hyperplane
that doesn't contain it is always $1$-dimensional. That's the beauty
of working in vector spaces, i.e.\ with homogenization!)
Let $v_i\in V_i\subset R$ be a non-zero vector. 
We claim that $v_1,\dots,v_n\in R$ give a \emph{vector representation}
of $\VV$ in~$R$.

For this, consider a line $\ell$ of the configuration $\VV$.
The corresponding facet $F^\ell\subset P$ (as described above) contains
the edges $e^i$ for $i\in I^0$, but \emph{not} the edges
$e^j$ for $j\in I^-\cup I^+$.
Thus if we intersect the hyperplane $H^\ell\subset\R^{2+n}$ with
$R$ we get a $2$-dimensional intersection that contains
$v_i$ ($i\in I^0$), but not $v_j$ ($j\in I^-\cup I^+$) --- otherwise
$H^\ell$ would contain $v_j$ as well as one of $\vb_j$ and $\vbb_j$, but not
the other one, which is impossible.
This completes the proof of the claim and of the theorem.
\end{proof}

\begin{corollary}
The Lawrence polytope $L(\VV_{11})$ derived from the extended pentagon configuration 
is a $13$-dimensional non-rational polytope with $22$ vertices:
It can be realized with vertex coordinates in~$\Q[\sqrt5]$, but not
with coordinates in~$\Q$.
\end{corollary}

\begin{oexercise}
 Construct a non-rational polytope with fewer vertices, and of smaller dimension.
\end{oexercise}

As a consequence of Richter-Gebert's work \cite{Rich4} we know that there
are even $4$-dimensional non-rational polytopes. Richter-Gebert's smallest
example has $33$ vertices.

\section*{Non-rational surfaces}

A \emph{polyhedral surface} $\Sigma\subset\R^3$ is composed
from convex polygons (triangles, quadrilaterals, etc.), which
are required to intersect nicely (that is, in a common edge,
a vertex, or not at all), and such that the union of all
polygons is homeomorphic to a closed surface (a sphere, a torus, etc.).

The basic ``gadget'' that we can use to build inherently non-rational
polyhedral surfaces from non-rational configurations is the 
``Toblerone torus'' --- a polyhedral nine-vertex torus built from nine
quadrilateral faces. As an abstract configuration, this is 
the surface that you get from a $3\times 3$ square by identifying
the points on opposite edges.
\[
\input 33torus.pstex_t
\]
You might think of such a torus as a polyhedral surface as built in
$3$-space from three Toblerone$^{\scriptsize\textregistered}$ (Swiss chocolate) boxes,
which are long thin triangular prisms; think of the triangles 
at the ends as tilted (which is true for the chocolate bars, but
not for their boxes).
\[
\includegraphics[height=50mm]{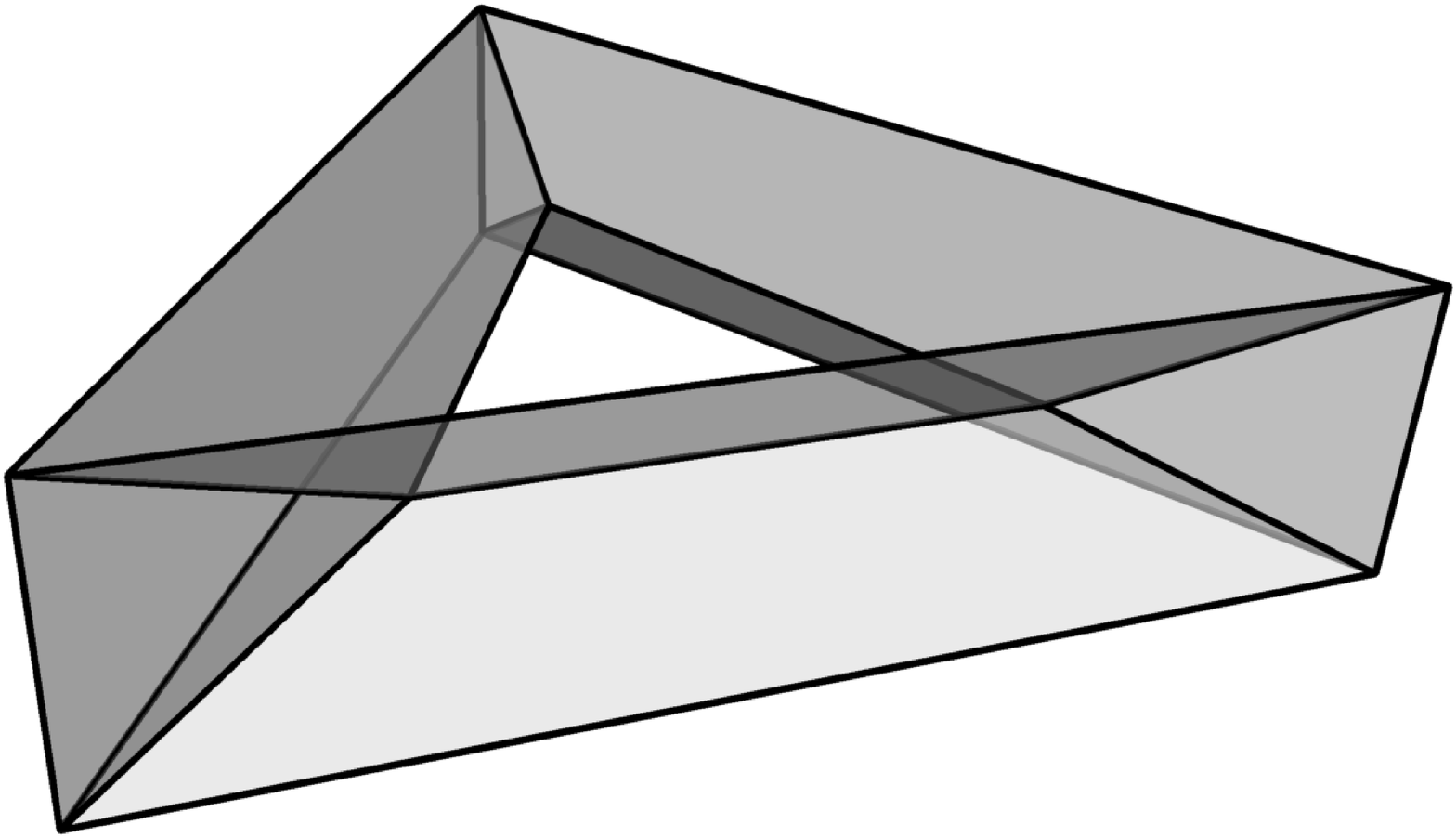}
\]

The key observation in this context is this:

\begin{lemma}[Simutis {\cite[Thm.~6, p.~43]{simutis77Diss}} 
\cite{gritzmann80-diss} \cite{RoerigPC}]
\label{lem:partial_toblerone}
  If you realize the toblerone torus in $\R^3$ with one quadrilateral missing,
then if the eight realized quadrilaterals are flat and convex, 
then the missing quadrilateral is necessarily flat, and 
it is necessarily convex.
\end{lemma}

The missing face of such an eight-quadrilateral Toblerone torus may be 
prescribed to be any given convex flat
quadrilateral in $3$-space: By projective
transformations on $3$-space, any convex flat quadrilateral 
can be mapped to any other one.

Now consider the following planar $9$-point configuration:
It consists of three black convex quadrilaterals
$adih$, $bfid$, $cgfe$, and three grey shaded quadrilaterals 
$bdhi$, $bfge$, and $cegi$.
\[
\psfrag{$a$}{$a$}
\psfrag{$b$}{$b$}
\psfrag{$c$}{$c$}
\psfrag{$d$}{$d$}
\psfrag{$e$}{$e$}
\psfrag{$f$}{$f$}
\psfrag{$g$}{$g$}
\psfrag{$h$}{$h$}
\psfrag{$i$}{$i$}
\includegraphics[height=40mm]{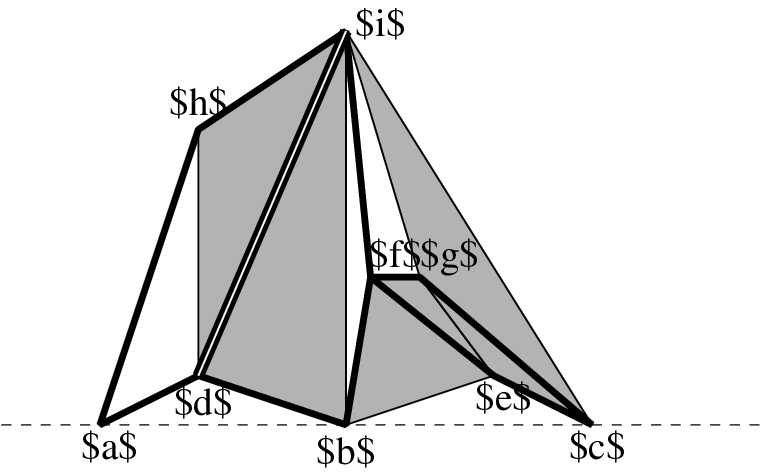}
\]
Think of this configuration as lying in a plane~$H$, 
and using projective transformations in $3$-space 
glue three toblerone tori with their missing 
faces onto the three black quadrilaterals, in such a way
that the three tori all come to lie on one side of the plane $H$.
Take three more Toblerone tori and glue them with their missing faces
onto the shaded grey quadrilaterals, on the other side of~$H$.

What you get is a partial polyhedral surface $S_{48}$, consisting
of $6\cdot 8=48$ convex quadrilaterals. It has
$9+6\cdot5=39$ vertices, among them the nine \emph{special} ones
which are labelled $a,b,\dots,i$. It could be
completed into a closed polyhedral surface by using
additional triangles and quadrilaterals, but let's
not do that for now.

\begin{lemma}[Brehm]
  In \emph{any} realization of the partial surface $S_{48}$,
  the $9$ special vertices $a,b,\dots,i$ lie in a plane.
\end{lemma}

\begin{proof}
  Indeed, by Lemma~\ref{lem:partial_toblerone}
  the six quadrilaterals are planar. It is easy to see
  that thus $a,h,d,b,i,f$ lie in one plane, and $c,e,f,g,b,i$ lie
  in one plane. Both planes contain $b,f,i$, and since they 
  cannot be collinear, the two planes coincide.
\end{proof}

Next, says Brehm, take three copies of the partial surface $S_{48}$,
and identify them in their copies of the vertices $a_j,b_j,c_j$,
(for $j=1,2,3$). This yields another partial surface $S_{144}$, 
consisting
of $3\cdot 48=144$ quadrilaterals and $3+3\cdot36=111$ vertices.

\begin{lemma}[Brehm]\label{lem9}
  In  \emph{any} realization of the partial surface $S_{144}$,
  the three special vertices $a,b,c$ lie on a line.
\end{lemma}

\begin{proof}
  Indeed, we know of three planes that the three vertices lie on.
  Two of these might coincide, where one $9$-point configuration
  could lie in the upper halfplane, and one in the lower half-plane,
  but the third configration then needs a different plane.
  Thus the three special vertices lie in the intersection of two~planes.
\end{proof}

From this it is quite easy to come up with, and to prove, Brehm's 
theorem: There are non-rational polyhedral surfaces!

\begin{theorem}[Brehm 1997/2007 
\cite{brehmOW2,brehm:_univer_theor_realiz_spaces_polyh_maps}]
  Glueing a copy of the partial surface $S_{144}$ into each of the
  collinear triples of the $11$-point pentagon configuration
  yields a partial surface that may be realized in 
  $\R^3$ with flat convex quadrilaterals.

  It may be completed into a closed, embedded polyhedral surface
  in $\R^3$ consisting of quadrilaterals and triangles, 
  all of whose vertex coordinates lie in~$\Q[\sqrt5]$.

  However, the partial surface (and hence the completed surface)
  does not have any rational realization.
\end{theorem}

Indeed, Lemma~\ref{lem9} represents already a major step on the
way to Brehm's universality theorem for polyhedral surfaces.

\section*{A glimpse of universality}

Following venerable traditions for example from Algebraic Geometry
(where one speaks of 
``moduli spaces'') it is natural and profitable to study not only
special realizations for discrete-geometric structures such as
configurations, polytopes or polyhedral surfaces, but also the 
\emph{space of all correct coordinatizations}, 
up to affine transformations, which is known
as the \emph{realization space} of the structure.

Why is this set a ``space'', what is its structure?  If we consider a
planar $n$-point configuration $\CC$, then a realization is given by
an ordered set of vectors $w_1,\dots,w_n$, which form the rows of a
matrix $W\in\R^{n\times2}$.  Thus a certain subset of the vector space
$\R^{n\times2}$ of all $2\times n$ matrices corresponds to ``correct''
realizations $W$ of ``our'' configuration~$\CC$.

In all three cases (configurations, polytopes, surfaces) the set of
correct realizations is a 
\emph{semi-algebraic set} (more precisely, a
primary semi-algebraic set defined over~$\Z$): It can be
described as the solution set of a finite system of polynomial
equations and strict inequalities in the coordinates, with 
integral coefficients.  For example, in the case of
configurations we specify for every triple $\vv_i,\vv_j,\vv_k$ 
that $\det(\vv_i,\vv_j,\vv_k)^2$ should be either zero or to be positive, which
amounts to a biquadratic equation resp.\ strict inequality in the
coordinates of $w_i,w_j$ and~$w_k$.

Any affine coordinate transformation corresponds to a 
column operation on the matrix $M\in\R^{n\times2}$.
So the realization space can be described as a quotient
of the set of all realization matrices by the action 
of the group of affine transformations. From this point of view,
it is not obvious that the realization space is a semi-algebraic
set. If, however, equivalently we fix an affine basis
(which in the plane means: fix the coordinates for three
non-collinear points to be the vertices of a specified triangle),
then this becomes clear.

\begin{proposition}[see Gr\"unbaum \cite{Gr1-2}]
  The realization space of any configuration, polytope or polyhedral
surface is a semi-algebraic set.
\end{proposition}

Semi-algebraic sets can be complicated: They can 
\begin{compactitem}[$\bullet$]
\item be empty, e.g.\ $\{x\in\R: x^2<0\}$,
\item be disconnected, e.g.\ $\{x\in\R: x^2>1\}$,
\item contain only irrational points,  $\{x\in\R: x^2=5\}$,
\end{compactitem}
etc. Indeed, this can easily be strengthened:
Semialgebraic sets have quite arbitrary homotopy types,
singularities, or need points from large extension fields of~$\Q$.

But can realization spaces for combinatorial structures be
so complicated and ``wild''?

It is a simple exercise to see that the realization space
for a convex $k$-gon $P\subset\R^2$ has a very
simple structure (equivalent to $\R^{2k-6}$).
Moreover, Steinitz \cite{Stei1,StRa} proved in 1910
that the realization space for every $3$-dimensional polytope
is equivalent to $\R^{e-6}$, where $e$ is the number of edges of~$P$. 
In particular, it contains rational points.
A similar result was also stated for general polytopes 
\cite{Robert} --- but it is not true.

A \emph{universality theorem} now mandates that the realization spaces
for certain combinatorial structures are as wild/\allowbreak
complicated/\allowbreak interesting/\allowbreak strange as arbitrary
semi-algebraic sets.

A blueprint is the universality theorem for oriented matroids by
Nikolai Mn\"ev, from which he also derived a universality
theorem for $d$-polytopes with $d+4$ vertices:

\begin{theorem}[Mn\"ev 1986 \cite{Mnev1,Mnev2}] For every
  semi-algebraic set $S\subset\R^N$ there is for some $d>2$ a
  $d$-polytope $P\subset\R^d$ with $d+4$ vertices whose realization
  space ${\mathcal R}(P)$ is ``stably equivalent'' to~$S$.
\end{theorem}

Such a result of course implies that there are non-rational polytopes,
that there are polytopes that have realizations that cannot be deformed
into each other (counterexamples to the ``isotopy conjecture''), etc.
(Here we consider the realization space of the whole polytope, 
not only of its boundary, that is, we are considering \emph{convex}
realizations only.)

To prove such a result, a first step is to 
find planar configurations that encode general polynomial systems;
the starting point for this are the ``von Staudt constructions'' 
\cite[2.~Heft]{staudt56:_beitr_geomet_lage} from the 19th century, which
encode addition and multiplication into incidence configurations.
This produces systematically examples such as the pentagon configuration
that we discussed.
Then one has to show that all real polynomial systems can be
brought into a suitable ``standard form'' (compare Shor \cite{Shor}),
develop a suitable concept of ``stably equivalent'' (compare Richter-Gebert \cite{Rich4}), and then go on.

In the last 20 years a number of substantial universality theorems
have been obtained, each of them technical, each of them a considerable 
achievement.
The most remarkable ones I know of today are the universality theorem for
$4$-dimensional polytopes by Richter-Gebert \cite{Rich4} (see also
G\"unzel \cite{Guenzel}), 
a universality for \emph{simplicial} polytopes by Jaggi et al.\ \cite{JMSW},
universality theorems for planar mechanical linkages by
Jordan \& Steiner \cite{jordan99:_config} 
and Kapovich \& Millson \cite{kapovich02:linkagage}, 
and the universality theorem for polyhedral surfaces by 
Brehm (to be
published~\cite{brehm:_univer_theor_realiz_spaces_polyh_maps}).

\section*{Four problems}

In the last forty years, there have been fantastic discoveries in the 
construction of non-rational examples, in the
study of rational realizations, and in the
development of universality theorems.
However, great challenges remain --- we take the
opportunity to close here with naming four.

\subsection*{Small coordinates}

According to Steinitz, every $3$-dimensional polytope can be
realized with rational, and thus also with integral vertex coordinates.
However, are there \emph{small} integral coordinates?
Can every $3$-polytope with $n$ vertices be realized with coordinates in
$\{0,1,2,\dots,p(n)\}$, for some polynomial $p(n)$?
Currently, only exponential upper bounds like
$p(n)\le {533}^{n^2}$ are known,
due to Onn \& Sturmfels~\cite{OnSt}, Richter-Gebert
\cite[p.~143]{Rich4}, and finally Rib\'o Mor \& Rote, see
\cite[Chap.~6]{Ribo-diss}.

\subsection*{The bipyramidal 720-cell}

It may well be that non-rational polytopes occur ``in nature''.
A good candidate is the ``first truncation'' of the regular
$600$-cell, obtained as the convex hull of the mid points of the 
edges of the $600$-cell, which has $600$ regular octahedra
and $120$ icosahedra as facets. 
This polytope was apparently already studied by Th.~Gosset in 1897;
it appears with notation 
$\left\{\begin{smallmatrix}3\\3,&5\end{smallmatrix}\right\}$
in Coxeter \cite[p.~162]{Coxeter}. Its dual,
which has $720$ pentagonal bipyramids as facets, is the   
$4$-dimensional \emph{bipyramidal 720-cell} of Gevay \cite{Gevay} \cite{Z89}. 
It is neither simple nor simplicial. 

Does this polytope (equivalently: its dual) have
a realization with rational coordinates?

\subsection*{Non-rational cubical polytopes}

As argued above, it is easy to see that all types of simplicial $d$-dimensional
polytopes can be realized with rational coordinates:
``Just perturb the vertex coordinates''. 
For \emph{cubical} polytopes, all of whose faces are
combinatorial cubes, there is no such simple argument.
Indeed, it is a long-standing open problem whether every
cubical polytope has a rational realization. This is true for 
$d=3$, as a special case of Steinitz's results. 
But how about cubical polytopes of dimension~$4$?
The boundary of such a polytope consists of combinatorial $3$-cubes;
its combinatorics is closely related with that of immersed cubical 
surfaces~\cite{Z91}.

On the other hand, if we impose the condition that the
cubes in the boundary have to be affine cubes --- 
so all $2$-faces are centrally symmetric --- then 
there are easy non-rational examples, namely the 
\emph{zonotopes} associated to non-rational configurations \cite[Lect.~7]{Z35}.

\subsection*{Universality for simplicial 4-polytopes}

There are universality theorems for
simplicial $d$-dimensional polytopes with $d+4$ vertices, and for 
$4$-dimensional polytopes.
But how about universality for simplicial $4$-dimensional polytopes?

The realization space for such a polytope is an \emph{open}
semi-algebraic set, so it certainly contains rational points, and
it cannot have singularities. 
One specific ``small'' simplicial $4$-polytope with $10$ vertices 
that has a combinatorial symmetry, but no symmetric realization, 
was described by Bokowski, Ewald \& Kleinschmidt in 1984 \cite{BEK};
according to Mn\"ev~\cite[p.~530]{Mnev1} 
and Bokowski \& Guedes de Oliveira \cite{BoGO} this example
does not satisfy the isotopy conjecture, that is,
the realization space is disconnected for this example.
Are there $4$-dimensional simplicial polytopes with
more/arbitrarily complicated homotopy types?

\end{multicols}
\begin{small}

\end{small}
\end{document}

%% file: pentagon1b.pstex_t
\begin{picture}(0,0)%
\includegraphics{pentagon1b.pstex}%
\end{picture}%
\setlength{\unitlength}{987sp}%
\begingroup\makeatletter\ifx\SetFigFont\undefined%
\gdef\SetFigFont#1#2#3#4#5{%
  \reset@font\fontsize{#1}{#2pt}%
  \fontfamily{#3}\fontseries{#4}\fontshape{#5}%
  \selectfont}%
\fi\endgroup%
\begin{picture}(14953,13002)(226,-9562)
\put(7326,2984){\makebox(0,0)[lb]{\smash{{\SetFigFont{10}{12.0}{\rmdefault}{\mddefault}{\updefault}{\color[rgb]{0,0,0}$p_2$}%
}}}}
\put(13641,-1711){\makebox(0,0)[lb]{\smash{{\SetFigFont{10}{12.0}{\rmdefault}{\mddefault}{\updefault}{\color[rgb]{0,0,0}$p_3$}%
}}}}
\put(11256,-9286){\makebox(0,0)[lb]{\smash{{\SetFigFont{10}{12.0}{\rmdefault}{\mddefault}{\updefault}{\color[rgb]{0,0,0}$p_4$}%
}}}}
\put(3336,-9391){\makebox(0,0)[lb]{\smash{{\SetFigFont{10}{12.0}{\rmdefault}{\mddefault}{\updefault}{\color[rgb]{0,0,0}$p_5$}%
}}}}
\put(6606,-6931){\makebox(0,0)[lb]{\smash{{\SetFigFont{10}{12.0}{\rmdefault}{\mddefault}{\updefault}{\color[rgb]{0,0,0}$p_7$}%
}}}}
\put(4161,-4846){\makebox(0,0)[rb]{\smash{{\SetFigFont{10}{12.0}{\rmdefault}{\mddefault}{\updefault}{\color[rgb]{0,0,0}$p_8$}%
}}}}
\put(5166,-1336){\makebox(0,0)[rb]{\smash{{\SetFigFont{10}{12.0}{\rmdefault}{\mddefault}{\updefault}{\color[rgb]{0,0,0}$p_9$}%
}}}}
\put(8556,-1216){\makebox(0,0)[lb]{\smash{{\SetFigFont{10}{12.0}{\rmdefault}{\mddefault}{\updefault}{\color[rgb]{0,0,0}$p_{10}$}%
}}}}
\put(9711,-4771){\makebox(0,0)[lb]{\smash{{\SetFigFont{10}{12.0}{\rmdefault}{\mddefault}{\updefault}{\color[rgb]{0,0,0}$p_{11}$}%
}}}}
\put(226,-1111){\makebox(0,0)[lb]{\smash{{\SetFigFont{10}{12.0}{\rmdefault}{\mddefault}{\updefault}{\color[rgb]{0,0,0}$p_6$}%
}}}}
\put(7056,-4456){\makebox(0,0)[lb]{\smash{{\SetFigFont{10}{12.0}{\rmdefault}{\mddefault}{\updefault}{\color[rgb]{0,0,0}$p_1$}%
}}}}
\end{picture}%

%% file: pentagon2b.pstex_t
\begin{picture}(0,0)%
\includegraphics{pentagon2b.pstex}%
\end{picture}%
\setlength{\unitlength}{987sp}%
\begingroup\makeatletter\ifx\SetFigFont\undefined%
\gdef\SetFigFont#1#2#3#4#5{%
  \reset@font\fontsize{#1}{#2pt}%
  \fontfamily{#3}\fontseries{#4}\fontshape{#5}%
  \selectfont}%
\fi\endgroup%
\begin{picture}(14215,12960)(151,-9544)
\put(10491,-7276){\makebox(0,0)[lb]{\smash{{\SetFigFont{10}{12.0}{\rmdefault}{\mddefault}{\updefault}{\color[rgb]{0,0,0}$\ell_7$}%
}}}}
\put(5251,-1336){\makebox(0,0)[rb]{\smash{{\SetFigFont{10}{12.0}{\rmdefault}{\mddefault}{\updefault}{\color[rgb]{0,0,0}$(1,-1,0)$}%
}}}}
\put(8556,-1216){\makebox(0,0)[lb]{\smash{{\SetFigFont{10}{12.0}{\rmdefault}{\mddefault}{\updefault}{\color[rgb]{0,0,0}$(1,1,0)$}%
}}}}
\put(12151,-1261){\makebox(0,0)[lb]{\smash{{\SetFigFont{10}{12.0}{\rmdefault}{\mddefault}{\updefault}{\color[rgb]{0,0,0}$(1,a,0)$}%
}}}}
\put(7326,2984){\makebox(0,0)[lb]{\smash{{\SetFigFont{10}{12.0}{\rmdefault}{\mddefault}{\updefault}{\color[rgb]{0,0,0}$v_2=(1,0,1)$}%
}}}}
\put(12151,-661){\makebox(0,0)[lb]{\smash{{\SetFigFont{10}{12.0}{\rmdefault}{\mddefault}{\updefault}{\color[rgb]{0,0,0}$v_3=$}%
}}}}
\put(11256,-9286){\makebox(0,0)[lb]{\smash{{\SetFigFont{10}{12.0}{\rmdefault}{\mddefault}{\updefault}{\color[rgb]{0,0,0}$v_4$}%
}}}}
\put(3336,-9391){\makebox(0,0)[lb]{\smash{{\SetFigFont{10}{12.0}{\rmdefault}{\mddefault}{\updefault}{\color[rgb]{0,0,0}$v_5$}%
}}}}
\put(6606,-6931){\makebox(0,0)[lb]{\smash{{\SetFigFont{10}{12.0}{\rmdefault}{\mddefault}{\updefault}{\color[rgb]{0,0,0}$v_7$}%
}}}}
\put(4161,-4846){\makebox(0,0)[rb]{\smash{{\SetFigFont{10}{12.0}{\rmdefault}{\mddefault}{\updefault}{\color[rgb]{0,0,0}$v_8$}%
}}}}
\put(4576,-736){\makebox(0,0)[rb]{\smash{{\SetFigFont{10}{12.0}{\rmdefault}{\mddefault}{\updefault}{\color[rgb]{0,0,0}$v_9=$}%
}}}}
\put(8551,-586){\makebox(0,0)[lb]{\smash{{\SetFigFont{10}{12.0}{\rmdefault}{\mddefault}{\updefault}{\color[rgb]{0,0,0}$v_{10}=$}%
}}}}
\put(9711,-4771){\makebox(0,0)[lb]{\smash{{\SetFigFont{10}{12.0}{\rmdefault}{\mddefault}{\updefault}{\color[rgb]{0,0,0}$v_{11}$}%
}}}}
\put(151,-1186){\makebox(0,0)[lb]{\smash{{\SetFigFont{10}{12.0}{\rmdefault}{\mddefault}{\updefault}{\color[rgb]{0,0,0}$v_6$}%
}}}}
\put(9336,-2161){\makebox(0,0)[lb]{\smash{{\SetFigFont{10}{12.0}{\rmdefault}{\mddefault}{\updefault}{\color[rgb]{0,0,0}$\ell_1$}%
}}}}
\put(8946,-6241){\makebox(0,0)[lb]{\smash{{\SetFigFont{10}{12.0}{\rmdefault}{\mddefault}{\updefault}{\color[rgb]{0,0,0}$\ell_3$}%
}}}}
\put(4686,-6376){\makebox(0,0)[rb]{\smash{{\SetFigFont{10}{12.0}{\rmdefault}{\mddefault}{\updefault}{\color[rgb]{0,0,0}$\ell_4$}%
}}}}
\put(6901,-4486){\makebox(0,0)[b]{\smash{{\SetFigFont{10}{12.0}{\rmdefault}{\mddefault}{\updefault}{\color[rgb]{0,0,0}$v_1=(1,0,-1)$}%
}}}}
\put(7041,434){\makebox(0,0)[lb]{\smash{{\SetFigFont{10}{12.0}{\rmdefault}{\mddefault}{\updefault}{\color[rgb]{0,0,0}$\ell_2$}%
}}}}
\put(10416,-3016){\makebox(0,0)[lb]{\smash{{\SetFigFont{10}{12.0}{\rmdefault}{\mddefault}{\updefault}{\color[rgb]{0,0,0}$\ell_5$}%
}}}}
\put(3321,-7471){\makebox(0,0)[rb]{\smash{{\SetFigFont{10}{12.0}{\rmdefault}{\mddefault}{\updefault}{\color[rgb]{0,0,0}$\ell_6$}%
}}}}
\put(10806,-3826){\makebox(0,0)[lb]{\smash{{\SetFigFont{10}{12.0}{\rmdefault}{\mddefault}{\updefault}{\color[rgb]{0,0,0}$\ell_7$}%
}}}}
\end{picture}%

%% file: lawrence_ext1.pstex_t
\begin{picture}(0,0)%
\includegraphics{lawrence_ext1.pstex}%
\end{picture}%
\setlength{\unitlength}{2368sp}%
\begingroup\makeatletter\ifx\SetFigFont\undefined%
\gdef\SetFigFont#1#2#3#4#5{%
  \reset@font\fontsize{#1}{#2pt}%
  \fontfamily{#3}\fontseries{#4}\fontshape{#5}%
  \selectfont}%
\fi\endgroup%
\begin{picture}(5649,2906)(1939,-6378)
\put(4289,-3676){\makebox(0,0)[lb]{\smash{{\SetFigFont{12}{14.4}{\rmdefault}{\mddefault}{\updefault}{\color[rgb]{0,0,0}$\vbb$}%
}}}}
\put(5586,-5956){\makebox(0,0)[rb]{\smash{{\SetFigFont{12}{14.4}{\rmdefault}{\mddefault}{\updefault}{\color[rgb]{0,0,0}$\VV$}%
}}}}
\put(4701,-4418){\makebox(0,0)[lb]{\smash{{\SetFigFont{12}{14.4}{\rmdefault}{\mddefault}{\updefault}{\color[rgb]{0,0,0}$\vb$}%
}}}}
\put(5176,-5086){\makebox(0,0)[lb]{\smash{{\SetFigFont{12}{14.4}{\rmdefault}{\mddefault}{\updefault}{\color[rgb]{0,0,0}$\vv$}%
}}}}
\end{picture}%

%% file: 33torus.pstex_t
\begin{picture}(0,0)%
\includegraphics{33torus.pstex}%
\end{picture}%
\setlength{\unitlength}{1302sp}%
\begingroup\makeatletter\ifx\SetFigFont\undefined%
\gdef\SetFigFont#1#2#3#4#5{%
  \reset@font\fontsize{#1}{#2pt}%
  \fontfamily{#3}\fontseries{#4}\fontshape{#5}%
  \selectfont}%
\fi\endgroup%
\begin{picture}(5210,4704)(1558,-5803)
\put(6076,-1786){\makebox(0,0)[lb]{\smash{{\SetFigFont{10}{12.0}{\rmdefault}{\mddefault}{\updefault}{\color[rgb]{0,0,0}$1$}%
}}}}
\put(6076,-5311){\makebox(0,0)[lb]{\smash{{\SetFigFont{10}{12.0}{\rmdefault}{\mddefault}{\updefault}{\color[rgb]{0,0,0}$1$}%
}}}}
\put(2251,-2911){\makebox(0,0)[rb]{\smash{{\SetFigFont{10}{12.0}{\rmdefault}{\mddefault}{\updefault}{\color[rgb]{0,0,0}$4$}%
}}}}
\put(2251,-4111){\makebox(0,0)[rb]{\smash{{\SetFigFont{10}{12.0}{\rmdefault}{\mddefault}{\updefault}{\color[rgb]{0,0,0}$7$}%
}}}}
\put(2251,-5311){\makebox(0,0)[rb]{\smash{{\SetFigFont{10}{12.0}{\rmdefault}{\mddefault}{\updefault}{\color[rgb]{0,0,0}$1$}%
}}}}
\put(2251,-1711){\makebox(0,0)[rb]{\smash{{\SetFigFont{10}{12.0}{\rmdefault}{\mddefault}{\updefault}{\color[rgb]{0,0,0}$1$}%
}}}}
\put(6076,-2911){\makebox(0,0)[lb]{\smash{{\SetFigFont{10}{12.0}{\rmdefault}{\mddefault}{\updefault}{\color[rgb]{0,0,0}$4$}%
}}}}
\put(6076,-4111){\makebox(0,0)[lb]{\smash{{\SetFigFont{10}{12.0}{\rmdefault}{\mddefault}{\updefault}{\color[rgb]{0,0,0}$7$}%
}}}}
\put(3601,-1411){\makebox(0,0)[b]{\smash{{\SetFigFont{10}{12.0}{\rmdefault}{\mddefault}{\updefault}{\color[rgb]{0,0,0}$2$}%
}}}}
\put(4801,-1411){\makebox(0,0)[b]{\smash{{\SetFigFont{10}{12.0}{\rmdefault}{\mddefault}{\updefault}{\color[rgb]{0,0,0}$3$}%
}}}}
\put(4801,-5686){\makebox(0,0)[b]{\smash{{\SetFigFont{10}{12.0}{\rmdefault}{\mddefault}{\updefault}{\color[rgb]{0,0,0}$3$}%
}}}}
\put(3601,-5686){\makebox(0,0)[b]{\smash{{\SetFigFont{10}{12.0}{\rmdefault}{\mddefault}{\updefault}{\color[rgb]{0,0,0}$2$}%
}}}}
\put(3676,-3211){\makebox(0,0)[lb]{\smash{{\SetFigFont{10}{12.0}{\rmdefault}{\mddefault}{\updefault}{\color[rgb]{0,0,0}$5$}%
}}}}
\put(4876,-3211){\makebox(0,0)[lb]{\smash{{\SetFigFont{10}{12.0}{\rmdefault}{\mddefault}{\updefault}{\color[rgb]{0,0,0}$6$}%
}}}}
\put(3676,-4411){\makebox(0,0)[lb]{\smash{{\SetFigFont{10}{12.0}{\rmdefault}{\mddefault}{\updefault}{\color[rgb]{0,0,0}$8$}%
}}}}
\put(4876,-4411){\makebox(0,0)[lb]{\smash{{\SetFigFont{10}{12.0}{\rmdefault}{\mddefault}{\updefault}{\color[rgb]{0,0,0}$9$}%
}}}}
\end{picture}%